\documentclass[]{article}
\usepackage[english]{babel}
\usepackage[normalem]{ulem}
\usepackage[all]{xy}
\usepackage{amsmath,amsthm,booktabs,color,dsfont,epsfig,graphicx,mathrsfs,multirow,scalefnt}
\usepackage{amsfonts, amssymb}
\usepackage{enumerate}

\newtheorem{theo}{Theorem}[section]
\newtheorem{cor}[theo]{Corollary}
\newtheorem{lem}[theo]{Lemma}

\newtheorem{prop}[theo]{Proposition}
\newtheorem{defn}[theo]{Definition}
\newtheorem{rmk}[theo]{Remark}
\newtheorem{ex}[theo]{Example}
\newcommand {\GG}{\mathcal{G}}

\newcommand{\Z}{\mathbb{Z}}
\newcommand{\N}{\mathbb{N}}

\newcommand{\s}{\sigma}

\newcommand{\A}{\mathbf{A}}

\newcommand{\G}{\mathcal{G}}
\newcommand{\h}{\mathcal{H}}

\title{Continuous shift commuting maps between ultragraph shift spaces}

\author{
\small{Daniel Gon\c{c}alves}\\
\footnotesize{UFSC -- Department of Mathematics}\\
\footnotesize{88040-900 Florian\'{o}polis - SC, Brazil}\\
\footnotesize{\texttt{daemig@gmail.com}}
\and
\small{Marcelo Sobottka}\\
\footnotesize{UFSC -- Department of Mathematics}\\
\footnotesize{88040-900 Florian\'{o}polis - SC, Brazil}\\
\footnotesize{\texttt{marcelo.sobottka@ufsc.br}}
}

\date{}

\begin{document}

\maketitle

\begin{abstract} Recently a generalization of shifts of finite type to the infinite alphabet case was proposed, in connection with the theory of ultragraph C*-algebras. In this work we characterize the class of continuous shift commuting maps between these spaces. In particular, we prove a Curtis-Hedlund-Lyndon type theorem and use it to completely characterize continuous, shift commuting, length preserving maps in terms of generalized sliding block codes.
\end{abstract}

\vspace{1.0pc}
MSC 2010: 37B10, 54H20, 37B15

\vspace{1.0pc}
Keywords: Symbolic dynamics, ultragraph edge shift spaces, infinite alphabets, Curtis-Hedlund-Lyndon theorem, generalized sliding block codes.

\bigskip
\hrule
\noindent
{\footnotesize\em This is a pre-copy-editing, author-produced PDF of an article accepted for publication in DCDS-A, following peer review. The definitive publisher-authenticated version {\em D. Gon\c{c}alves and  M. Sobottka. Continuous shift commuting maps between ultragraph shift spaces. Disc. and Cont. Dynamic. Systems (2019), 39, 2, 1033-1048; doi:10.3934/dcds.2019043 }, is available online at: http://aimsciences.org//article/doi/10.3934/dcds.2019043.}
\hrule
\bigskip
% ---------------------------- INTRODUCTION ------------------------------------

\section{Introduction}

The generalization of the idea of a subshift of finite type to the case of a countable alphabet, called a countable-state topological Markov chain, is a natural one to make and comes up in various contexts, including problems in magnetic recording, see \cite{Petersen}. Countable-state topological Markov chains have also been studied in papers like \cite{Fiebig2001, Fiebig2003, FiebigFiebig1995, FiebigFiebig2005, kitchens1997}, to mention a few. Although the subject of intense research, the development of results for infinite alphabet shift spaces, that parallel the symbolic dynamics of shifts over finite alphabets, has challenged researchers over the years. The lack of compactness (or local compactness) in the spaces considered account to many results in usual symbolic dynamics failing. For example, it is shown in \cite{Petersen, Sal} that for shifts with a countable alphabet (defined via product topology) the entropy of a factor may increase.

In \cite{Ott_et_Al2014} Ott, Tomforde and Willis proposed a definition of a compact shift space that is related to C*-algebra theory. Building from these ideas, and on work of Webster (see \cite{Webster}), a generalization of shifts of finite type to the infinite alphabet case was proposed recently in \cite{GRISU}. The construction proposed in \cite{GRISU} takes the shift space as the boundary path space of an ultragraph (ultragraphs are combinatorial objects that generalize direct graphs). The idea is that the boundary path space is the the spectrum of a certain Abelian subalgebra of the ultragraph C*-algebra. In a similar way, in the finite alphabet case, a Markov shift is the spectrum of an abelian subalgebra of the associated Cuntz-Krieger algebra, see \cite{CK}. Although the theory of shift spaces defined in \cite{GRISU} is still in its infancy, there has been already applications to KMS states associated to ultragraph C*-algebras, see \cite{CG}, and to the diagonal-preserving isomorphism problem of ultragraph C*-algebras, see \cite{CRST, GRISU}.

Continuous shift commuting maps form the main class of maps studied in symbolic dynamics. The importance of this class of maps arises from the fact that given two shift spaces, viewed as topological dynamical systems with the correspondent shift maps, a shift commuting homeomorphism from one shift space to the other is by definition a topological conjugacy between the dynamical systems. For example, in \cite{GRISU} shift morphisms (continuous, shift commuting maps) between shift spaces were studied, in connection with isomorphism of the associated ultragraph C*-algebras.

For shift spaces over finite alphabets, the Curtis-Hedlund-Lyndon Theorem gives a complete characterization of the class of continuous shift commuting maps: Such class of maps corresponds to the class of sliding block codes, that is, corresponds to the class of maps which have bounded local rules\footnote{When sliding block codes are defined from a shift space onto itself they are named cellular automata and, as proposed by von Neumann (see \cite{vonNeumann}), are topological dynamical systems that serve as models for self-reproducing and self-organizing systems.} (see \cite[Chap. 6]{LindMarcus}).

For infinite-alphabet shift spaces (with the product topology) it was proved that continuous shift commuting maps correspond to generalized sliding block codes, that are maps which have local rules, but their local rules are not necessarily bounded (see \cite{GS}). In particular, uniformly continuous shift commuting maps correspond to sliding block codes in the classical sense of maps with bounded local rules (see \cite{Ceccherini-Silberstein--Coornaert}). In the Ott-Tomforde-Willis context, it was showed in \cite{GSS0} that there exist continuous shift commuting maps that are not generalized sliding block codes, and there exist generalized sliding block codes that are not continuous shift commuting maps. Furthermore, in \cite{GSS0} a complete characterization of the intersection of the class of continuous shift commuting maps with the class of generalized sliding block codes was given.

In this paper we provide a characterization of continuous shift commuting maps between the shift spaces defined in \cite{GRISU} (see Theorem~\ref{CSC}). In particular, we describe the connection between continuous shift commuting maps and generalized sliding block code (see Theorem~\ref{general-CHL-T}). As a result we completely characterize continuous, shift commuting, length preserving maps in terms of generalized sliding block codes (see Corollary~\ref{length-preserving_SBC}). Before we proceed to the main section (Section~3), we present a review of the ultragraph shift spaces given in \cite{GRISU} in Section~2 below.

\section{Background}\label{Background}

In this section we recall some background on ultragraphs and the shift spaces associated to them. We also set notation. Throughout this paper $\N$ denotes the set of positive integers.

\subsection{Ultragraphs}

Ultragraphs were introduced by Tomforde in \cite{Tom3} as the correct object to unify the study of graph and Cuntz-Krieger algebras (via ultragraph C*-algebras). Since their introduction ultragraphs have been used in connection with both dynamical systems and C*-algebra theory (see \cite{GLR, KMST, TomSimple} for example). Recently ultragraphs have become a key object in the study of infinite alphabet shift spaces, see \cite{GRultra, GRISU}. In this section we recall the main definitions and set up notation, following closely the notions introduced in \cite{MarreroMuhly, Tom3}.

\begin{defn}\label{def of ultragraph}
An \emph{ultragraph} is a quadruple $\mathcal{G}=(G^0, \mathcal{G}^1, r,s)$ consisting of a set of vertexes $G^0$, a set of edges $\mathcal{G}^1$, a map $s:\mathcal{G}^1 \to G^0$, and a map $r:\mathcal{G}^1 \to P(G^0)\setminus \{\emptyset\}$, where $P(G^0)$ stands for the power set of $G^0$.
\end{defn}

\begin{defn}\label{def of mathcal{G}^0}
Let $\mathcal{G}$ be an ultragraph. Define $\mathcal{G}^0$ to be the smallest subset of $P(G^0)$ that contains $\{v\}$ for all $v\in G^0$, contains $r(e)$ for all $e\in \mathcal{G}^1$, and is closed under finite unions and non-empty finite intersections (a characterization of  $\mathcal{G}^0$ in terms of intersections and unions of ranges of edges can be found in \cite[Lemma~2.12]{Tom3}).
\end{defn}

Let $\mathcal{G}$ be an ultragraph. A \textit{finite path} in $\mathcal{G}$ is either an element of $\mathcal{G}^{0}$ or a sequence of edges $\alpha=(\alpha_i)_{i=1}^k$ in $\mathcal{G}^{1}$, where
$s\left(  \alpha_{i+1}\right)  \in r\left(  \alpha_{i}\right)  $ for $1\leq i\leq k$. The set of finite paths in $\mathcal{G}$ is denoted by $\mathcal{G}^{\ast}$.

If we write $\alpha=(\alpha_i)_{i=1}^k$, then the length $\left|  \alpha\right|  $ of
$\alpha$ is just $k$. The length $|A|$ of a path $A\in\mathcal{G}^{0}$ is
zero. We define $r\left(  \alpha\right)  =r\left(  \alpha_{k}\right)  $ and
$s\left(  \alpha\right)  =s\left(  \alpha_{1}\right)  $. For $A\in\mathcal{G}^{0}$,
we set $r\left(  A\right)  =A=s\left(  A\right)  $.

An \textit{infinite path} in $\mathcal{G}$ is an infinite sequence of edges $\gamma=(\gamma_i)_{i\geq 1}$ in $\prod \mathcal{G}^{1}$, such that
$s\left(  \gamma_{i+1}\right)  \in r\left(  \gamma_{i}\right)  $ for all $i$. The set of
infinite paths  in $\mathcal{G}$ is denoted by $\mathfrak{p}^{\infty}_\G$. The length $\left|  \gamma\right|  $ of $\gamma\in\mathfrak{p}^{\infty}_\G$ is defined to be $\infty$, and we define $s(\gamma)=s(\gamma_1)$. A vertex $v$ in $\mathcal{G}$ is
called a \emph{sink} if $\left|  s^{-1}\left(  v\right)  \right|  =0$ and is
called an \emph{infinite emitter} if $\left|  s^{-1}\left(  v\right)  \right|
=\infty$. %We say that a vertex $v$ is a \emph{singular vertex} if it is either
%a sink or an infinite emitter.

We set $\mathfrak{p}^0_\G:=\mathcal{G}^{0}$ and, for $n\geq1$, we define
$\mathfrak{p}^{n}_\G:=\{\left(  \alpha,A\right)  :\alpha\in\mathcal{G}^{\ast
},\left\vert \alpha\right\vert =n,$ $A\in\mathcal{G}^{0},A\subseteq r\left(
\alpha\right)  \}$, and $$\mathfrak{p}_\G:=\bigcup\limits_{n\geq 0}\mathfrak{p}^{n}_\G.$$ We specify that $\left(  \alpha,A\right)  =(\beta,B)$ if,
and only if, $\alpha=\beta$ and $A=B$. We
define the length of $\left(  \alpha,A\right)\in\mathfrak{p}_\G$ as $|\left(\alpha,A\right)|:=|\alpha|$.
 We call $\mathfrak{p}_\G$ the \emph{ultrapath space}
associated with $\mathcal{G}$ and the elements of $\mathfrak{p}_\G$ are called
\emph{ultrapaths}. Each $A\in\mathcal{G}^{0}$ is regarded as an ultrapath of length zero and can be identified with the pair $(A,A)$. We embed the set of finite paths $\GG^*$ in $\mathfrak{p}$ by sending $\alpha$ to $(\alpha, r(\alpha))$.  We extend the range map $r$ and the source map $s$ to
$\mathfrak{p}_\G$ by the formulas, $r\left(  \left(  \alpha,A\right)  \right)
=A$, $s\left(  \left(  \alpha,A\right)  \right)  =s\left(  \alpha\right)
$ and $r\left(  A\right)  =s\left(  A\right)  =A$.\\

Given $\alpha=(\alpha_i)_{i=1}^k$ and $\beta=(\beta_i)_{i=1}^\ell$ in $\GG^*$ with $s(\beta)\in r(\alpha)$ we define the concatenation of $\alpha$ with $\beta$ as $\alpha\beta:=(\alpha_{1}\ldots \alpha_{k}\beta_{1}\ldots \beta_{\ell})\in \GG^*$. Given $\alpha\in \GG^*$ we say that $\alpha'\in \GG^*$ is a prefix, or initial segment, of $\alpha$ if either $\alpha'=\alpha$ or $\alpha=\alpha'\beta$ for some $\beta\in \GG^*$.

 Given $x\in\mathfrak{p}_\G$ and $y\in \mathfrak{p}_\G\cup\mathfrak{p}^{\infty}_\G$ such that $s(y)\subseteq r(x)$ (if $|y|=0)$ or $s(y)\in r(x)$ (if $|y|\geq 1$), we define the concatenation of $x$ and $y$ (and denote it as $xy$) as follows:

 \begin{equation}
 \begin{array}
 [c]{lcl}%
 x=A & \Rightarrow & xy := y;\\
 x=(\alpha,A) \text{ and } y=B& \Rightarrow & xy := (\alpha,B);\\
 x=(\alpha,A) \text{ and } y=(\beta,B)& \Rightarrow & xy := (\alpha\beta,B);\\
 x=(\alpha,A) \text{ and } y=(y_i)_{i\geq 1}\ldots& \Rightarrow & xy := (\alpha_1\ldots\alpha_{|\alpha|} y_1y_2y_3\ldots)
\end{array}
 \label{specify}%
 \end{equation}
 Given $x\in\mathfrak{p}_\G\cup\mathfrak{p}^{\infty}_\G$, we say that $x$ has $x'\in \mathfrak{p}_\G$ as a prefix, or initial segment, if
 $x=x'y$, for some $y\in\mathfrak{p}_\G\cup\mathfrak{p}^{\infty}_\G$.

\begin{defn}
\label{infinte emitter} For each subset $A$ of $G^{0}$, let
$\varepsilon\left(  A\right)  $ be the set $\{ e\in\mathcal{G}^{1}:s\left(
e\right)  \in A\}$. We shall say that a set $A$ in $\mathcal{G}^{0}$ is an
\emph{infinite emitter} whenever $\varepsilon\left(  A\right)  $ is infinite.
\end{defn}

\subsection{Ultragraph shift spaces}

In this section we recall the definition of a shift space associated to an ultragraph, as introduced in \cite{GRISU}. Since \cite{GRISU} only deals with ultragraphs without sinks we make the same assumption here.

{\bf Throughout assumption:} From now on all ultragraphs in this paper are assumed to have no sinks.

Before we define the topological space associated to an ultragraph we need the following definition.

\begin{defn}\label{minimal} Let $\GG$ be an ultragraph and $A\in \GG^0$. We say that $A$ is a minimal infinite emitter if it is an infinite emitter that contains no proper subsets (in $\GG^0$) that are infinite emitters. For a finite path $\alpha$ in $\GG$, we say that $A$ is a minimal infinite emitter in $r(\alpha)$ if $A$ is a minimal infinite emitter and $A\subseteq r(\alpha)$. We denote the set of all minimal infinite emitters in $r(\alpha)$ by $M_\alpha$, and define $$\mathfrak{p}_{\G min}:=\{(\alpha,A)\in\mathfrak{p}_\G: A\in M_\alpha\},$$ and $$\mathfrak{p}^0_{\G min}:= \mathfrak{p}_{\G min} \cap \mathfrak{p}^0_\G.$$

\end{defn}

To standardize the notation with previous work on sliding block codes (see \cite{GSS, GSS0} for example) we let
 $$X_\G^{fin}:=\{(x_i)_{i\geq 1}: (x_i)_{i\geq 1}=(\alpha_1\ldots\alpha_k AA\ldots)\text{ with } (\alpha_1\ldots\alpha_k,A)\in \mathfrak{p}_{\G min}\};$$
and let $X_\G^{inf}:=\mathfrak{p}^{\infty}_\G.$
%$$X_\G^{0}:=\{(x_i)_{i\geq 1}: \exists A \in \mathfrak{p}^0_{\G min} \text{ such that } x_i=A \ \forall i\geq 1 \}.$$

\begin{rmk} Notice that $X_\G^{fin}$ can be embedded in $\mathfrak{p}_\G$, via the map $\iota$ that takes $(\alpha_1\ldots\alpha_k AA\ldots)$ to $(\alpha_1\ldots\alpha_k,A)$. So we can translate concepts defined in $\mathfrak{p}_\G$ to $X_\G^{fin}$. For example, $y\in \mathfrak{p}_\G$ is a prefix of $x$ in $X_\G^{fin}$ iff it is a prefix of $\iota(x)$ in $\mathfrak{p}_\G$.
\end{rmk}

\begin{defn} Let $\G$ be an ultragraph.
We denote the set of sequences of the form $(AAA\ldots)$ in $X_\G^{fin}$ by $X_\G^{0}$. Elements of $X_\G^{0}$ are called 0-sequences and we set their length as zero. A sequence of the form $x=(\alpha_1\ldots\alpha_n AA\ldots)\in X_\G^{fin}$ is called a finite-sequence, or an $n$-sequence, (and we set its length as $|x|:=n$). Finally, a sequence of the form $x=(x_i)_{i\geq 1}\in X_\G^{inf}$ is called an infinite-sequence and we set its length as $|x|:=\infty$.
\end{defn}

As a topological space, the (shift) space associated to an ultragraph $\G$ is the set $$X_\G:= X_\G^{fin}\ \cup\ X_\G^{inf},$$
endowed with the topology generated by {\em generalized cylinders}, which are sets of the form:
%Abaixo j\'{a} foi definido
%We say that $A\in\mathfrak{p}^0_\G$ is a prefix, or initial segment, of a sequence $(x_i)_{i\geq 1}\in X_\G$ if and only if $s(x_1)\in A$ or $s(x_1)%%%%%\subset A$. For $(\alpha_1\ldots\alpha_n,A)\in \mathfrak{p}_\G$, we say that it is a prefix of a sequence $(x_i)_{i\geq 1}\in X_\G$ if $x_i=\alpha_i$ %for all $1\leq i\leq n$ and $s(x_{n+1})\in A$ or $s(x_{n+1})\subset A$.\\

\begin{equation}\label{generalizedcylinders}D_{\textbf{y},F}:=\{x\in X_\G: \textbf{y}\text{ is prefix of } x \text{ and } x_{|\textbf{y}|+1}\notin F\},\end{equation}
where $\textbf{y}\in \mathfrak{p}_\G$ and $F$ is a finite (possibly empty) subset of $\varepsilon(r(\textbf{y}))$. When $F=\emptyset$ we use the short notation $D_{\textbf{y}}:=D_{\textbf{y},F}$.

%%%%%% DESTRINCHANDO A EQUA\c{C}\~{A}O ACIMA

% \begin{itemize}
% \item If $\textbf{y}=B\in \mathfrak{p}^0_\G$, then
% $$D_{\textbf{y},F}=D_{B,F}= \{A\in\mathfrak{p}^{0,fin}: A\subset B\}\cup\{y\in\mathfrak{p}^{\infty}_\G\cup\mathfrak{p}^{fin}:y_1\in \varepsilon(B)\setminus F\}.$$
%
% \item If $\textbf{y}=(\beta,B)\in \mathfrak{p}^{n}_\G$, then
% $$D_{\textbf{y},F}=D_{(\beta,B),F}= \{(\beta,A)\in\mathfrak{p}^{fin}: A\subset % B\}\cup\{y\in\mathfrak{p}^{\infty}_\G\cup\mathfrak{p}^{fin}:y_i=\beta_i\ \forall i=1,\ldots,n \text{ and } y_{n+1}\in \varepsilon(B)\setminus F\}.$$\end{itemize}

%%%%%%%%%%%%%%%%%%%%%%%%%%%%%%%%%%%%

We remark that the generalized cylinders form a countable basis of clopen (but not necessarily compact) sets for a metrizable topology on $X_\G$ (see \cite{GRISU} for details, including conditions for local compactness of $X_\G$). Furthermore:

 \begin{itemize}
 \item If $x = (x_i)_{i\geq 1} \in X_\G^{inf}$ then a neighbourhood basis for $x$ is given by $$\{D_{(x_1 \ldots x_n, r(x_n))}: n\geq 1 \}; $$

 \item If $x=(\alpha AAA\ldots) \in X_\G^{fin}$ then a neighbourhood basis for $x$ is given by $$ \{D_{(\alpha, A),F}: F\subset \varepsilon\left(  A\right), |F|<\infty \};$$

%\item If $x=(AAA\ldots) \in X_\G^0$ then a neighbourhood basis for $x$ is given by $$ \{D_{A,F}: F\subset \varepsilon\left(  A\right), |F|<\infty \}.$$
\end{itemize}

For our work the description of convergence of sequences in $X_\G$ is important. We recall it below:

\begin{prop}\label{convseq} Let $\{x^n\}_{n=1}^{\infty}$ be a sequence of elements in $X_\G$, where $x^n = (\gamma^n_1\ldots \gamma^n_{k_n}, A_n)$ or $x^n = \gamma_1^n \gamma_2^n \ldots$, and let $x \in X_\G$.
\begin{enumerate}[(a)]
\item If $|x|= \infty$, say $x=\gamma_1 \gamma_2 \ldots$, then $\{x^n\}_{n=1}^{\infty}$ converges to $x$ if, and only if, for every $M\in \N$ there exists $N\in \N$ such that $n>N$ implies that $|x^n|\geq M$ and $\gamma^n_i= \gamma_i$ for all $1\leq i \leq M$.

\item If $|x|< \infty$, say $x=(\gamma_1 \ldots \gamma_k, A)$, then $\{x^n\}_{n=1}^{\infty}$ converges to $x$ if, and only if, for every finite subset $F\subseteq \varepsilon\left(  A\right)$ there exists $N\in \N$ such that $n > N$ implies that $x^n = x$ or $|x^n|> |x|$, $\gamma^n_{|x|+1} \in \ \varepsilon\left(  A\right)\setminus F$, and $\gamma^n_i = \gamma_i$ for all $1 \leq i \leq |x|$.
\end{enumerate}
\end{prop}

We define the shift map $\s:X_\G\to X_\G$ in the usual way:
$$\s\big((x_i)_{i\geq 1}\big)=(x_{i+1})_{i\geq 1}.$$

%\begin{rmk} Notice that if $|x| = \infty$ then $|\sigma(x)| = \infty$, if $1\leq |x|<\infty$ then $|\sigma(x)| = |x|-1$, and if $|x| = 0$ then $|\sigma(x)| = 0$.
%\end{rmk}

The shift map is not continuous at points of length zero. This will play an important role in our results. We have the following result regarding continuity of $\s$.

\begin{prop}\label{Rmk:shift_invariance_cylinders}
The shift map $\sigma : X_\G \rightarrow X_\G$ is continuous at all points of $X_\G$ with length greater than zero. %In addition, if $|x|\geq 1$ then there exists an open set $U$ that contains no elements of length zero such that $x \in U$, $\sigma(U)$ is an open subset of $X$, and $\sigma|_U : U \to \sigma(U)$ is a homeomorphism.
Furthermore, if $y\in\mathfrak{p}_\G$ and $|y|>0$ then $\s(D_{y,F})=D_{\s(y),F}$ (but this is not necessarily true if $|y|=0$).
\end{prop}

Next we recall the definition of the shift space.

\begin{defn} Let $\GG$ be an ultragraph. The {\em one-sided shift space} associated to $\GG$ is the pair $(X_\G,\sigma)$, where $X_\G$ and $\s$ are as defined above (with $X_\G$ viewed as a topological space). We will often refer to the space $X_\G$ with the understanding that the map $\sigma$ is attached to it.
\end{defn}

%\begin{rmk} Notice that if $\GG$ is a finite graph then there are no infinite emitters and $X_\G$ is the usual infinite path space of the graph. So the definition coincides with the usual definition of an edge shift.
%\end{rmk}

For use in the next sections we introduce the following definition.

\begin{defn} Let $\G$ be an ultragraph. The {\em alphabet} of the shift $X_\G$ is defined as the set $\A_\G$ of all the symbols that can appear in some sequence of $X_\G$, that is, $$\A_\G:=\G^1\cup \mathfrak{p}^0_{\G min}.$$

\end{defn}

\section{Continuous shift invariant maps}\label{sec:shift_invariant_maps}

The characterization of continuous, shift commuting maps is the main goal of this section (and of the paper). Before we prove our main results (in Subsection~\ref{CHL-Theo}), we need to develop a few auxiliary results. As mentioned before, we are under the assumption that all ultragraphs have no sinks.

\subsection{Shift commuting maps}

In this subsection we study shift commuting maps between shift spaces. We give a characterization of such maps below.

\begin{prop}\label{prop:shift_invariant_maps}
Let $\G$ and $\h$ be ultragraphs, and let $X_\G$ and $X_\h$ be their respective associated ultragraph shifts. A map $\Phi:X_\G\to X_\h$ is shift commuting (i.e. $\Phi \circ \s = \s \circ \Phi$) if, and only if, there exists a family of sets $\big\{C_a\big\}_{a\in\A_{\h}}$, which is a partition of $X_\G$, such that for all $x\in X_\G$ and $n\geq 1$ we have
    \begin{equation}\label{eq:shift_invariant_maps}\bigl(\Phi(x)\bigr)_n=\sum_{a\in \A_{\h}}a\mathbf{1}_{C_a}\circ\sigma^{n-1}(x), \end{equation}
where $\mathbf{1}_{C_a}$ is the
characteristic function of the set $C_a$ and $\sum$ stands for the symbolic sum.

\end{prop}

\begin{proof}

Suppose that $\Phi$ is shift commuting. Given $a\in\A_{\h}$,  let $C_a:=\Phi^{-1}\big(D_{(a,r(a))})$ if $a\in\h^1$, and let $C_a:=\Phi^{-1}\big(\{A\})$ if $a=A\in\mathfrak{p}^0_{\h min}$. It is straightforward that $\big\{C_a\big\}_{a\in\A_{\h}}$ is a partition of $X_\G$.

Now, for each $x\in X_\G$, to determine $(\Phi(x))_1$ it is only necessary to know what set $C_a$ contains $x$, that is, $\bigl(\Phi(x)\bigr)_1=\sum_{a\in\A_{\h}}a\mathbf{1}_{C_a}(x)$. Therefore, since for each $n\geq 1$ we have
$\Phi\circ\s^n=\s^n\circ\Phi$, it follows that
$$\bigl(\Phi(x)\bigr)_n=\bigl(\sigma^{n-1}(\Phi(x))\bigr)_1=\bigl(\Phi(\sigma^{n-1}(x))\bigr)_1=\sum_{a\in\h}a\mathbf{1}_{C_a}(\sigma^{n-1}(x)).$$

For the converse, suppose that $\Phi$ is given by \eqref{eq:shift_invariant_maps}. To check that $\Phi$ is shift commuting we just need to check that, for all $x\in X_\G$ and $n\geq 1$, we have $\Big(\Phi\big(\s(x)\big)\Big)_n=\Big(\s\big(\Phi(x)\big)\Big)_n$. This follows from the following computation.
    $$\Big(\Phi\big(\s(x)\big)\Big)_n
    =\sum_{a\in \A_\h}a\mathbf{1}_{C_a}\circ\sigma^{n-1}\big(\s(x)\big)
    =\sum_{a\in \A_\h}a\mathbf{1}_{C_a}\circ\sigma^{n}(x)
    =\big(\Phi(x)\big)_{n+1}
    =\Big(\s\big(\Phi(x)\big)\Big)_n.$$

\end{proof}

%\begin{rmk} Notice that if $\phi$ is a shift commuting map (described as in Proposition~\ref{prop:shift_invariant_maps}) then, for each $A\in  \mathfrak{p}^0_{\h min}$, the set $C_A$ is shift invariant (that is $\s(C_A)\subseteq C_A$). Indeed, note that if $\phi(x)=A$ then $\phi (\sigma (x))= \sigma (\phi(x)) = A$.
%\end{rmk}

The following results will be useful in the next section.

\begin{lem}\label{sliding block code->finite seq goes to finite seq}
Let $\Phi:X_\G\to X_\h$ be a shift commuting map and $(AA\ldots)\in X^0_\G$. If $|\Phi(AA\ldots)|=0$ then the image of every $x=(x_1x_2\ldots x_nA\ldots)\in X^{fin}_\G$ under $\Phi$ is a finite sequence in $X_\h$ with length no greater than $|x|$.
\end{lem}
\begin{proof}

Let $(AA\ldots)\in X^0_\G$ and suppose that $\Phi((AA\ldots)) = (BB\ldots)\in X^0_\h$. Let $x:=(x_1x_2\ldots x_nA\ldots)\in X^{fin}_\G$,  then
$$\sigma^{|x|} \circ \Phi(x) = \Phi \circ \sigma^{|x|} (x) = \Phi (AA\ldots) = (BB\ldots).$$

\end{proof}

\begin{lem}\label{auxiliar} Let $\Phi:X_{\GG}\to X_{\h}$ be a shift commuting map. Then $\phi(e_1e_2e_3\ldots)$ = $a \phi (e_2 e_3 \ldots)$, where $a\in A_{X_{\h}}$. The same result holds for finite sequences.
\end{lem}

The next two results follow as in Section~ 3.1 of \cite{GSS}.

\begin{prop}\label{sliding block code->preserves_period}
If $\Phi:X_\G\to X_\h$ is a shift commuting map and $x\in X_\G$ is a sequence with period $p\geq 1$ (that is, such that $\s^p(x)=x$) then $\Phi(x)$ also has period $p$.
\end{prop}

\begin{cor}\label{sliding block code->emptysequence goes to constant}
If $\Phi:X_\G\to X_\h$ is a shift commuting map  then, for all $(AA\ldots)\in X^0_\G$, we have that $\Phi(AA\ldots)$ is a constant sequence (that is, $\Phi(AA\ldots)=(ddd\ldots)$ for some $d\in \A_\h$).
\end{cor}

We end this section by proving that for a shift commuting map $\Phi:X_\G\to X_\h$, described in terms of characteristic functions of a partition $\big\{C_a\big\}_{a\in\A_{\h}}$ of $X_\h$ as in Proposition~\ref{prop:shift_invariant_maps}, the sets associated to the elements of length zero are shift invariant.

\begin{cor}\label{cor:maps_between_ultragraph_shifts->C_A_are_shift_invariant}
Let $\G$ and $\h$ be two ultragraphs, and $X_\G$ and $X_\h$ be the associated ultragraph shifts, respectively. Let $\Phi:X_\G\to X_\h$ be a shift commuting map and $\big\{C_a\big\}_{a\in\A_{\h}}$ be the partition of  $X_\h$ given in Proposition~\ref{prop:shift_invariant_maps}. Then, for all $A\in\mathfrak{p}^0_{\h min}$, we have that $\s(C_A)\subset C_A$, that is, $C_A$ is shift invariant.
\end{cor}

\begin{proof}

Given $A\in\mathfrak{p}^0_{\h min}$, let $x\in C_A$ and $y:=\Phi(x)\in X_\h$. Since $y_1=\big(\Phi(x)\big)_1=A \in X_\h$ it follows that $y_i=A$ for all $i\geq 1$. Thus, for all $i\geq 1$, it follows that $A=y_i=\Big(\s^{i-1}\big(\Phi(x)\big)\Big)_1=\Big(\Phi\big(\s^{i-1}(x)\big)\Big)_1$, which means that $\s^{i-1}(x)\in C_A$.

\end{proof}

\subsection{Generalized sliding block codes}\label{sec:Sliding block codes}

In this subsection we recall the concept of generalized sliding block codes, which rely on the notion of finitely defined sets. We also present examples in the ultragraph setting. We start with the definition of blocks.

Let $\G$ be an ultragraph. For each $n\geq 1$, let
$$B_n(X_\G):=\{(a_1\ldots a_n)\in (\A_\GG)^n:\ \text{there exists } x\in X_\G,\ i\geq 1, \text{ such that } x_{i+j-1}=a_j\text{ for all } j=1,\ldots,n\}$$
be the set of all {\em blocks of length $n$} in $X_\G$.\\

%We single out the blocks of length one and use the notation $\LA:= B_1(X_\G)\setminus\{A\in \GG^0: A \text{ is a minimal infinite emitter} \}$ -- this is the set of all symbols used by sequences of $X_\G$, or the {\em letters} of $X_\G$.

The {\em language} of $X_\G$ is \begin{equation}\label{language}B(X_\G):=\bigcup_{n\geq 1}B_n(X_\G).\end{equation}

\begin{rmk} Notice that while a finite sequence $(\alpha AAA\ldots) \in X_\G^{fin}$ has length $|\alpha|$ the block $(\alpha A\ldots A) $, where we repeat $n$ times the symbol $A$, has length $|\alpha| + n$.
\end{rmk}

Before we can introduce finitely defined sets we need the notion of pseudo cylinders.

\begin{defn} A pseudo cylinder in a shift space $X_\G$ is a set of the form
$$[b]_{k}^\ell:=\{(x_i)_{i\in\N}\in X_\G: (x_{k}\ldots x_\ell)=b\},$$
where $1\leq k\leq\ell$ and $b\in B(X_\G)$. We also assume that the empty set is a pseudo cylinder.
\end{defn}

%We say that the pseudo cylinder $[b]_{k}^\ell$ has anticipation $\ell$.
%We say that the pseudo cylinder $[b]_{k}^\ell$ has memory  $K:=-\min \{0,k\}$ and anticipation $L:=\max\{0,\ell\}$. Note that if $\Lambda$ is a one-sided shift space, then any pseudo cylinder of $\Lambda$ has memory equal to zero. We adopt the convention that the empty set is a pseudo cylinder of $\Lambda$ whose memory and anticipation are zero.

We remark that in the context of shift spaces with the product topology pseudo cylinders are equivalent to cylinders. On the other hand, for ultragraph shift spaces (and also in the context of the shift spaces studied in \cite{GSS0,GSS1} and \cite{Ott_et_Al2014}), a pseudo cylinder is not necessarily an open set. However, as we will see in Proposition~\ref{lem:generalized_cylinders_are_finitely_defined}, a generalized cylinder, and its complement, can always be written as union of pseudo cylinders, that is, a generalized cylinder is a finitely defined set, accordingly to the following:

\begin{defn} Given $C\subset X_\G$, we say that $C$ is a {\em finitely defined} in $X_\G$ if both $C$ and $C^c$ can be written as unions of pseudo cylinders. More precisely, $C$ is finitely defined if there exist two collections of pseudo cylinders in $X_\G$, namely $\{[b^i]_{k_i}^{\ell_i}\}_{i\in I}$ and $\{[d^j]_{m_j}^{n_j}\}_{j\in J}$,  such that
\begin{equation*}
C=\bigcup_{i\in I} [b^i]_{k_i}^{\ell_i}\qquad \text{and}\qquad
C^c= \bigcup_{j\in J}[d^j]_{m_j}^{n_j}.
\end{equation*}
\end{defn}

\begin{rmk} Intuitively, a finitely defined set $C$ in $X_\G$ is a set such that, given $x\in X_\G$, we can `decide' whether it belongs (or not) to $C$ by knowing a finite quantity of its coordinates.
\end{rmk}

%In this case we say that $C$ has memory  $K:=-\inf_{i\in I} \{0,\ k_i\}$ and anticipation $L:=\sup_{i\in I}\{0,\ \ell_i\}$.

The empty set and $X_\G$ itself are trivial examples of finitely defined sets in $X_\G$. Other examples are:

\begin{ex}\label{finitexample} Let $\GG$ be an ultragraph and let $Z$ be a subset of $X_\G^0$. Then $Z$ is a finitely defined set. On the other hand, suppose that there exist $\gamma = e_1 e_2 \ldots \in \mathfrak{p}^{\infty}_\G$ such that $|s(e_i)|\geq 2$ for each $i$. Then $\{\gamma \}^c$, which can be written as a countable union of generalized cylinder sets, is not finitely defined. % (notice that the existence of $\gamma$ with the mentioned property depends on the ulgragraph $\GG$).
\end{ex}
 \begin{proof}

 Notice that $Z= \displaystyle \bigcup_{AA\ldots \in Z} [AA]_1^2$ and $Z^c = \displaystyle \bigcup_{a\in \A_{\GG} \setminus Z} [a]^1_1$ (where on the second union we use the identification of $AA\ldots $ in $X_\G^0$ with $A\in \mathfrak{p}_{\G min}^0$).

For the second part, notice that $\{\gamma \}$ can not be written as an union of pseudo cylinders.

 \end{proof}

 As we already mentioned, generalized cylinder sets are finitely defined. We prove this below.

\begin{prop}\label{lem:generalized_cylinders_are_finitely_defined} Let $\G$ be an ultragraph and $X_\G$ be the associated ultragraph shift space. Then, for all $\textbf{y}\in \mathfrak{p}_\G$ and all finite set $F\subset \varepsilon(r(\textbf{y}))$, the generalized cylinder $D_{\textbf{y},F}$ is a finitely defined set.
\end{prop}
\begin{proof}

Let $\textbf{y}\in \mathfrak{p}_\G$, and let $F\subset \varepsilon(r(\textbf{y}))$ be a finite set. If $\textbf{y}=A\in \mathfrak{p}_{\G}^0$, then it follows that
$$D_{A,F}=\bigcup_{e \in \varepsilon(A)\setminus F}[e]_1^1 \ \ \bigcup \bigcup_{B\subset A,\ B\in \mathfrak{p}^0_{\G min}}[B]_1^1$$
and
$$D_{A,F}^c=\bigcup_{e \in \varepsilon(A)^c\cup F}[e]_1^1 \ \ \bigcup  \bigcup_{B\not \subset A,\ B\in \mathfrak{p}^0_{\G min}}[B]_1^1.$$

If $\textbf{y}=(\gamma_1\ldots\gamma_n,A)\in\mathfrak{p}_\G \setminus \mathfrak{p}_{\G}^0$, then
$$D_{(\gamma_1\ldots\gamma_n,A),F}=\bigcup_{e \in \varepsilon(A)\setminus F}[\gamma_1\ldots\gamma_ne]_1^{n+1} \ \ \bigcup  \bigcup_{B\subset A,\ B\in \mathfrak{p}^0_{\G min}}[\gamma_1\ldots\gamma_nB]_1^{n+1}$$
and
%$$D_{(\gamma_1\ldots\gamma_n,A),F}^c=
%\bigcup_{(\alpha_1\ldots\alpha_n)\neq(\gamma_1\ldots\gamma_n)}[\alpha_1\ldots\alpha_n]_1^{n} \ \ \bigcup
%\bigcup_{e \in F\cup(r(\gamma_n)\setminus \varepsilon(A))}[\gamma_1\ldots\gamma_ne]_1^{n+1} \bigcup
%\bigcup_{B\not\subset A,\ B\in \mathfrak{p}^0_{\G min}}[\gamma_1\ldots\gamma_nB]_1^{n+1}.$$

$$D_{(\gamma_1\ldots\gamma_n,A),F}^c=
\bigcup[\alpha_1\ldots\alpha_n]_1^{n} \bigcup
[\gamma_1\ldots\gamma_ne]_1^{n+1} \bigcup
[\gamma_1\ldots\gamma_nB]_1^{n+1}, $$ where the unions in the right side range over $(\alpha_1\ldots\alpha_n)\neq(\gamma_1\ldots\gamma_n)$, $e \in F\cup(r(\gamma_n)\setminus \varepsilon(A))$, and $B\not\subset A,\ B\in \mathfrak{p}^0_{\G min}$, respectively.

\end{proof}

Following the same outline of the proof of Proposition~3.6 in \cite{GSS1}, one can prove that:

\begin{prop}\label{prop:finite_union_fin_def_sets} Finite unions and finite intersections of finitely defined sets are also finitely defined.
\end{prop}

\begin{rmk} In general infinite unions or intersections of finitely defined are not finitely defined sets. Thus infinite unions of generalized cylinders need not be finitely defined sets.
\end{rmk}

%The next lemma gives sufficient conditions under which infinite unions of pseudo cylinders (and thus of finitely defined sets) are finitely defined sets.

%\begin{lem}\label{lem:countable_union_fin_def_sets} Let $V,W\subset X_\G$ be such that both $V$ and $W$ are countable unions of pseudo cylinders, and $V%\cup W$ is finitely defined set. Then both $V$ and $W$ are finitely defined sets.
%\end{lem}

%\begin{proof}
%We will only to prove that $V$ is a finitely defined set, since the proof for $W$ is analogous. For this, we just need to prove that $V^c$ is a union of pseudo cylinders. Note that $$V^c=(V\cup W)^c\cup (W\setminus V).$$

%From the hypotheses of the lemma, $V\cup W$ is a finitely defined set, which implies that $(V\cup W)^c$ is a union of pseudo cylinders, and $W$ is a union of pseudo cylinders. Hence, from Proposition \ref{prop:dif_unions_pseudo_cylinders} it follows that $W\setminus V$ is also a union of pseudo cylinders, which concludes the proof.

%\end{proof}

% To better situate the reader in the discussion of the remark above, we present a couple of examples, of finitely
% defined sets below.
%

%
%

Now that we have a good understanding of finitely defined sets we can define generalized sliding block codes.

\begin{defn}\label{defn:sliding block code}
Let $\G$ and $\h$ be two ultragraphs and let $X_\G$ and $X_\h$ be the associated ultragraph shift spaces, respectively. We say that a map $\Phi:X_\G\to X_\h$ is a {\em generalized sliding block code} if for all $x\in X_\G$ and $n\geq 1$ it follows that
\begin{equation*}\bigl(\Phi(x)\bigr)_n=\sum_{a\in \A_{\h}}a\mathbf{1}_{C_a}\circ\sigma^{n-1}(x), \end{equation*}
where $\{C_a\}_{a\in \A_\h}$ is a partition of $X_\G$ by finitely defined sets.
\end{defn}

Note that, from Proposition \ref{prop:shift_invariant_maps}, generalized sliding block codes are shift commuting maps. Notice also that a generalized sliding block code can be interpreted as a map with a (possible unbounded) local rule, that is, a map such that to determine $\big(\Phi(x)\big)_j$ one just need to know the configuration of $x$ in a finite window around $x_j$ (but this window can vary). When the local rule is bounded, in the sense that the window around $x_j$ is always the same, the classical notion of sliding block codes is recovered. %In terms of the family $\{C_a\}_{a\in \A_\h}$ that define $\Phi$, to be a sliding block code means that there exist $L\geq 1$ such that each finitely defined set $C_a$ is the union of pseudo cylinder $[b_1\ldots b_n]_1^\ell$ with $\ell\leq L$.\\

We also remark that in the case of classical shift spaces over a finite alphabet, generalized sliding block codes are the same as classical sliding block codes (see \cite{GS}) and hence they coincide with the continuous shift commuting maps (see \cite{LindMarcus}). In the case of shift spaces over an infinite alphabet, with the product topology, generalized sliding block codes also always coincide with the continuous shift commuting maps \cite{GS}. %, and they will be classical sliding block codes if and only if they are uniformly continuous (see\cite{Ceccherini-Silberstein--Coornaert}).
For Ott-Tomforde-Willis shift spaces, it was showed in \cite{GSS0} that there exists generalized sliding block codes that are not continuous, and sufficient and necessary condition under which generalized sliding block codes coincide with continuous shift-invariant maps were presented. In our setting, if $X_\G^{fin} = \emptyset$ (for example, in the case of a row finite graph), then the topology on $X_\G$ coincides with the product topology and hence generalized sliding block codes coincide with continuous, shift commuting maps (as in \cite{GS}). In the next section we characterize continuous, shift commuting maps in $X_\G$.

\subsection{Continuous shift commuting maps and generalized sliding block codes}\label{CHL-Theo}

In this section we study continuous, shift commuting maps and their connection with generalized sliding block codes. We start by proving a result regarding continuity of shift commuting maps on $X_\G^{\inf}$ (for which we need the following lemma).

\begin{lem}\label{fdcontainscyl} Let $C$ be a finitely defined set in $X_\G$. If $x\in C\cap X_\G^{inf}$ then there exists a generalized cylinder $D$ such that $x\in D$ and $D\subseteq C$.
\end{lem}
\begin{proof}

Let $x\in C\cap X_\G^{inf}$, where $C$ is finitely defined in $X_\G$. Since $|x|=\infty$, $C$ must contain a pseudo cylinder of the form $[x_k\ldots x_l]_k^l$, where $x_j \in \G^1$ for $j=l,\ldots, k$. But pseudo cylinders of the aforementioned type can be written as an union of generalized cylinders and hence the result follows.
\end{proof}

\begin{prop}\label{continfity} Let $\Phi:X_\G\to X_\h$ be a shift commuting map, characterized in terms of partitions $\big\{C_a\big\}_{a\in\A_{\h}}$, as in Proposition~\ref{prop:shift_invariant_maps}. Suppose that $\Phi$ is continuous on $X^{inf}_{\G}\cap \Phi^{-1}(X^0_{\h})$. Furthermore, suppose that for all $a\in \A_{\h}\setminus \mathfrak{p}^0_{\h min}$, and all $x\in C_a \cap X^{inf}_{\G}$, there exists a cylinder $D$ such that $x\in D \subseteq C_a$. Then $\phi$ is continuous on $X^{inf}_{\G}$.
\end{prop}

%\begin{prop} Let $\Phi:X_\G\to X_\h$ be a shift commuting map, characterized in terms of partitions $\big\{C_a\big\}_{a\in\A_{\h}}$, as in Proposition~\ref{prop:shift_invariant_maps}. Suppose that $\Phi$ is continuous on every $x\in X^{inf}_{\G}$ such that $|\Phi(x)|=0$. Furthermore, suppose that for all $a\in \A_{\h}\setminus \mathfrak{p}^0_{\h min}$, and all $x\in C_a \cap X^{inf}_{\G}$, there exists a cylinder $D$ such that $x\in D \subseteq C_a$. Then $\phi$ is continuous on $X^{inf}_{\G}\cap Phi^{-1}($.
%\end{prop}
\begin{proof}

Let $x \in X^{inf}_{\G} \setminus \Phi^{-1}(X^0_{\h})$, say $x= \alpha_1 \alpha_2 \ldots$, and let $(x^n)$ be a sequence in $X_\G$ converging to $x$.

Suppose that $|\Phi(x)|= \infty$, say $\Phi(x)=\beta_1 \beta_2 \ldots$. Given $K>0$ we have to show that there exists $N>0$ such that $\Phi(x^n)_j = \Phi(x)_j$, for all $j=1,\ldots,K$, $n>N$.

Notice that, for $j=1,\ldots,K$, $\sigma^{j-1}(x) \in C_{\beta_j}\cap X^{inf}_{\G}$. Hence, by hypothesis, there exists a cylinder $D_j$ such that $\sigma^{j-1}(x) \in D_j \subseteq C_{\beta_j}$. Since $\sigma^j(x^n)$ converges to $\sigma^j(x)$, there exists $N_j$ such that $\sigma^{j-1}(x^n) \in C_{\beta_j}$ for all $n>N_j$. Therefore $\Phi(x^n)_j = \Phi(x)_j$ for all $j=1,\ldots,K$ and $n>\text{max}\{N_1 \ldots N_{K}\}$.

Now suppose that $|\Phi(x)| < \infty$, say $\Phi(x)=(\beta_1\ldots\beta_k BB\ldots) \in X_{\h}^{fin}$. Since $\s^{i-1}(x)\in C_{\beta_i}$ for each $i=1,\ldots,k$, there exists $N_1>0$ such that, for all $n>N_1$, we have $\s^{i-1}(x^n)\in C_{\beta_i}$. Hence $(\Phi(x^n))_i = \beta_i$ for all $i=1, \ldots, k$, and $n>N_1$. Notice that $\s^k(x)\in C_B\cap X^{inf}_{\G}$. Since $\s^k (x^n)$ converges to $\s^k (x)$, $\Phi$ is shift commuting, and by hypothesis $\Phi$ is continuous on $\s^k(x)$, it follows that $\Phi(x^n)$ converges to $\Phi(x)$.

\end{proof}

\begin{cor}\label{prop:sliding_bock_codes_continuous_on_infinity_sequences} If $\Phi:X_\G\to X_\h$ is a generalized sliding block code then it is shift commuting and continuous on $X_\G^{inf}$.
\end{cor}

\begin{proof}

It follows from Proposition~\ref{prop:shift_invariant_maps} that $\Phi$ is shift commuting.

Let $\{C_a\}_{a\in \A_\h}$ be the partition that defines $\Phi$, as in Definition~\ref{defn:sliding block code}. Let $x\in X_\G^{inf}$. By Lemma~\ref{fdcontainscyl} if $x\in C_a$, for some $a\in \A_{\h}\setminus \mathfrak{p}^0_{\h min}$, there exists a cylinder $D_a$ such that $x\in D_a \subseteq C_a$. If $x\in \Phi^{-1}(X^0_{\h})$ then $x\in C_B$, for some $B \in \mathfrak{p}^0_{\h min}$. Since $C_B$ is a finitely defined set, Lemma~\ref{fdcontainscyl} implies again that there exists a cylinder $D$ such that $x\in D \subseteq C_B$. So $\Phi$ is locally constant in $x$ and hence continuous on $x$. Continuity of $\Phi$ on $X_\G^{inf}$ now follows from Proposition~\ref{continfity}.

\end{proof}

Next we characterize continuous shift commuting maps.

\begin{theo}\label{CSC}
Let $\Phi:X_{\GG}\rightarrow X_{\h}$ be a map. If $\Phi$ is continuous and shift commuting then $\Phi$ is a map given by $$\bigl(\Phi(x)\bigr)_n=\sum_{a\in \A_{\h} }a\mathbf{1}_{C_a}\circ\s^{n-1}(x), \text{ for all } n\geq 1,$$ where

\begin{enumerate}
\item $\{C_a\}_{a\in \A_{\h}}$ is a partition of $X_\G$ such that, for each $a\in  \A_{\h}\setminus \mathfrak{p}^0_{\h min}$, the set $C_a$ is a (possibly empty) union of generalized cylinders of $X_{\GG}$;

\item if %$\Phi((AA\ldots)) = (BB\ldots) \in X_\h^0$ for some $(AA\ldots) \in X^0_{\G}$, and
$\Phi(\alpha_1\ldots\alpha_k AA\ldots) = (\beta_1\ldots \beta_l BB\ldots)$, for some $(\alpha_1\ldots\alpha_k AA\ldots) \in X_{\G}^{fin}$ (in particular $l\leq k$ by Lemma~\ref{sliding block code->finite seq goes to finite seq}), then for every neighbourhood $D_{B,F}$ of $B$ there exists a cylinder $D_{(\s^{l}(\alpha_1\ldots\alpha_k),A),F'}$ such that $\Phi(D_{(\s^{l}(\alpha_1\ldots\alpha_k),A),F'})\subseteq D_{B,F}$;

\item if $|\Phi(AA\ldots)|>0$, say $\Phi(AA\ldots) = (ddd\ldots)$ for some $(AA\ldots) \in X_\G^0$, then for all $M>0$ there exists a cylinder $D_{A,F}$ such that $\s^i(D_{A,F}) \subseteq C_d$ for all $i=0, 1, \ldots, M$.

\end{enumerate}

Under the additional hypothesis that $\Phi$ is continuous on $X^{inf}_{\G}\cap \Phi^{-1}(X^0_{\h})$ the converse of the statement above also holds.
\end{theo}

\begin{proof}

Let $\Phi:X_{\GG}\rightarrow X_{\h}$ be a continuous and shift commuting map.

From Proposition~\ref{prop:shift_invariant_maps} we have that, for all $x\in X_\G$ and $n\geq 1$, $$\bigl(\Phi(x)\bigr)_n=\sum_{a\in \A_\h}a\mathbf{1}_{C_a}\circ\s^{n-1}(x),$$ where for each $e \in \h^1$ we have $C_e := \Phi^{-1}(D_{(e,r(e))})$, and for each $B\in \mathfrak{p}^0_{\h min}$ we have $C_B= \Phi^{-1}(B)$.

Notice that $\{C_a\}_{a\in \A_{\h}}$ is a partition of $X_{\GG}$. Furthermore, notice that each $C_e$ is clopen and, since the generalized cylinders in $X_{\GG}$ form a countable basis, each $C_e$ can be written as a countable union of generalized cylinder sets. Therefore Item~i. is satisfied.

To check that Item~ii. holds notice that if $\Phi(\alpha_1\ldots \alpha_k A A \ldots) = (\beta_1\ldots \beta_l B B \ldots)$ then, by the continuity of $\Phi$, for every neighbourhood $D_{B,F}$ of $B$ there exists a cylinder $D_{(\alpha_1\ldots \alpha_k,A),F'}$ such that $\Phi (D_{(\alpha_1\ldots \alpha_k,A),F'}) \subseteq D_{(\beta,B),F}$. Hence, since $\Phi$ is shift commuting, we get that $\Phi(D_{(\s^{l}(\alpha_1\ldots \alpha_k),A),F'})\subseteq D_{B,F}$.

Next suppose that there exists  $(AA\ldots) \in X_{\G}^0$ such that $|\Phi(AA\ldots)|>0$, say $\Phi(AA\ldots) = (ddd\ldots)$. Notice that $|\Phi(AA\ldots)|=\infty$. Let $M>0$ and $\alpha = (d\ldots d)$ be a block of length $M+1$. Then $(AA\ldots)\in \Phi^{-1}(D_{(\alpha,r(d))})$ and $\Phi^{-1}(D_{(\alpha,r(d))})$ is open. Therefore there exists a cylinder $D_{A,F}\subseteq \Phi^{-1}(D_{(\alpha,r(d))})$. Let $1\leq i \leq M$ and $x\in D_{A,F}$. Then $\Phi (\s^i(x)) = \s^i(\Phi(x))$ and, since $\Phi(x)\in D_{(\alpha,r(d))}$, we get that $\s^i(x) \in C_d$ and Item~iii. is proved.

Now suppose that $\Phi$ is continuous on $X^{inf}_{\G}\cap \Phi^{-1}(X^0_{\h })$. Under this condition we show the converse of the theorem.

Assume that $\Phi$ is given by $\bigl(\Phi(x)\bigr)_n=\sum_{a\in A_{\GG_2} }a\mathbf{1}_{C_a}\circ\s^{n-1}(x)$, and Items i. to iii. above are satisfied. By Proposition~\ref{prop:shift_invariant_maps} we have that $\Phi$ is shift commuting. We prove that it is also continuous.

Notice that, by Proposition~\ref{continfity}, $\Phi$ is continuous on $X^{inf}_{\GG}$. Therefore we only need to show continuity on $X^{fin}_{\GG}$.

Let $(x^n)$ be a sequence in $X_{\GG}$ that converges to $x\in X^{fin}_{\GG}$. We divide the proof in two cases.

{\bf Case 1:} If $|x|=0$.

Then $x = (AA\ldots)$ for some $(AA\ldots)\in X_\G^0$. If $\Phi((AA\ldots)) = (BB\ldots)$ for some $(BB\ldots)\in X_\h^0$ then, by Item~ii., $\Phi(x^n)$ converges to $\Phi(x)$. Suppose that $\Phi((AA\ldots)) = ddd\ldots$, with $|\Phi((AA\ldots))|=\infty$. Let $M\in \N$. By Item~iii. there exists a cylinder $D_{A,F}$ such that $\s^i(D_{A,F}) \subseteq C_d$ for all $i=0, 1,\ldots, M-1$. Since $x_n$ converges to $A$, there exists $N>0$ such that $x_n \in D_{A,F}$ for every $n>N$. Therefore $\Phi(x_n)_i =d$ for all $i=1,\ldots, M$ and hence $\Phi(x^n)$ converges to $\Phi(x)$.

{\bf Case 2:} If $0<|x|<\infty$, say $x= \alpha_1\ldots \alpha_k A A \ldots$.

 By the description of converge of sequences in $X_{\GG_1}$ we may assume, without loss of generality, that $|x^n|\geq k$ for all $n$.

  Suppose that $\Phi((AA\ldots)) = (B'B'\ldots)$, where $(B'B'\ldots) \in X_\h^0$. By Lemma~\ref{sliding block code->finite seq goes to finite seq} we have that $\Phi(x) = (\beta_1 \ldots \beta_l BB\ldots)$, where $l\leq k$. %By Corollary~\ref{sliding block code->emptysequence goes to constant} we have that $\Phi(A)$ is a constant sequence.
Notice that $\Phi((AA\ldots)) = \phi( \s^{k}(x)) = \s^{k} (\Phi(x)) = (BB\ldots)$ and hence $B= B'$. Fix a natural number $j$ such that $1\leq j \leq l$. Note that $\sigma^{j-1}(x) \in C_{\beta_j}$, and hence there is a generalized cylinder $D_j$ such that $\sigma^{j-1}(x)\in D_j \subseteq C_{\beta_j}$. Since $(\sigma^{j-1}(x^n))$ converges to $\sigma^{j-1}(x)$, there exists an $N_j$ such that, for all $n>N_j$, $\sigma^{j-1}(x^n) \in D_j$ and hence $\bigl(\Phi(x^n)\bigr)_j=\beta_j$ (so the j-entry of $\Phi(x^n)$ is $\beta_j$). Now, let $D_{(\beta_1 \ldots \beta_l ,B),F}$ be a generalized cylinder set containing $(\beta_1 \ldots \beta_l, B)$. Then $D_{B,F}$ is a generalized cylinder set containing $B$. Pick a cylinder $D_{(\s^l(\alpha),A),F'}$ such that $\Phi(D_{(\s^l(\alpha),A),F'})\subseteq D_{B,F}$ (from item ii. of hypothesis). By Proposition~\ref{Rmk:shift_invariance_cylinders}, we have that $(\sigma^l(x^n))$ converges to $(\s^l(\alpha)AA\ldots)$ and hence there exists $N_{l+1}$ such that, for all $n> N_{l+1}$, $\sigma^l(x^n)\in D_{(\s^l(\alpha),A),F'} $. Taking $N$ as the maximum among $N_1, \ldots N_{l+1}$, and using Lemma~\ref{auxiliar}, we have that $\Phi(x^n) \in D_{(\beta_1 \ldots \beta_l ,B),F}$ for all $n>N$. Therefore $\Phi(x_n)$ converges to $\Phi(x)$.

Now suppose that $\Phi((AA\ldots)) = ddd\ldots$,  with $|\Phi((AA\ldots))| = \infty$ (so $d\in \h^1)$. By Lemma~\ref{auxiliar} we have that $\Phi(x) = \beta_1 \beta_2 \ldots $, where $\beta_i \in \h^1$ for $i=1.. |x|$, and $\beta_i = d \in \h^1$ for $i>|x|$. Notice that $\s^{j-1}(x) \in C_{\beta_j}$ for each $j \in \N$, and hence, by Item~i., there are generalized cylinders $D_j$ such that $\sigma^{j-1}(x)\in D_j \subseteq C_{\beta_j}$ for all $j\leq |x|$. Since $x^n$ converges to $x$ we have that $\s^j(x_n)$ converges to $\s^j(x)$ for all $j\leq |x|$. Therefore we can find $N_1$ such that, for all $n>N_1$ and for all $j=1, \ldots, |x|$, it holds that $\sigma^{j-1}(x^n) \in D_j$ and hence $\bigl(\Phi(x^n)\bigr)_j=\beta_j$. Let $M> |x|$. Take a cylinder $D_{A,F}$ as in Item~iii., that is, such that $\s^i(D_{A,F}) \subseteq C_d$ for all $i=0, 1, \ldots, M-|x|$. Since $\s^{|x|}(x_n)$ converges to $\s^{|x|}(x) = (AA\ldots)$ we have that there exists $N>N_1$ such that, for all $n>N$, $\s^{|x|}(x_n)\in D_{A,F}$. Hence $\s^{i+|x|}(x_n)\in D_{A,F}$ for all $i=0, 1, \ldots, M-|x|$ and therefore  $\bigl(\Phi(x^n)\bigr)_j=\beta_j$ for $j=1\ldots M$. We conclude that $\Phi(x_n)$ converges to $\Phi(x)$.

\end{proof}

\begin{rmk} When $X^{inf}_{\G}\cap \Phi^{-1}(X^0_{\h})$ is empty, the above theorem is a complete characterization of shift commuting maps. This is the case of maps such that $\Phi(X_{\GG_1}^{inf}) \subseteq X_{\GG_2}^{inf}$ or that preserve length (when dealing with infinite alphabet shift spaces the hypothesis that shift commuting maps preserve length is common, see for example \cite{GR, GRultra, GRISU, Ott_et_Al2014}).
\end{rmk}

Next we connect shift commuting maps with generalized sliding block codes.

\begin{theo}\label{general-CHL-T}Let $X_\G$ and $X_\h$ be two ultragraph shift spaces.
Suppose that $\Phi:X_\G\to X_\h$ is a map such that for all $B\in\mathfrak{p}^0_{\h min}$ the set $C_B:=\Phi^{-1}(BBB\ldots)$ is a finitely defined set. Then $\Phi$ is continuous and shift commuting if, and only if, $\Phi$ is a generalized sliding block code given by $\bigl(\Phi(x)\bigr)_n=\sum_{a\in \A_\h}a\mathbf{1}_{C_a}\circ\s^{n-1}(x)$ where:
\begin{enumerate}
\item\label{theo:general-CHL-Theo:condition_1} For all $a\in \h^1$, the set $C_a$ is a (possibly empty) union of generalized cylinders of $X_\G$;

\item\label{theo:general-CHL-Theo:condition_2} If $(\bar x_1\ldots \bar x_{|\bar x|} AAA\ldots)\in X^{fin}_\G$ is such that $\Phi(\bar x_1\ldots \bar x_{|\bar x|} AAA\ldots)=(BBB\ldots)\in X_\h^0$, then:
 \begin{enumerate}[a.]
 \item There exists a finite subset $F \subseteq \varepsilon(A)$ such that, for all $e \in \varepsilon(A)\setminus F$, if $x\in X_\G$ satisfies $x_i=\bar x_i$ for all $i=1,\ldots,|\bar x|$, and $x_{|\bar x|+1}=e$, then $\big(\Phi(x)\big)_1=B$ or $(\Phi(x)\big)_{1}\in\varepsilon(B)$, i.e., $\Phi(x) \in D_{B}$;

\item For all $x\in X_\G$ with $x_i=\bar x_i$ for $i=1,\ldots,|\bar x|$, $x_{|\bar x|+1}\in\varepsilon(A)$, and $\big(\Phi(x)\big)_1\in\varepsilon(B)$, the set $$A_x:=\{g\in\varepsilon(A): \text{ there exists } y\in X_\G \text{ with } y_i=\bar x_i \text{ for } i=1,\ldots,|\bar x|, \ y_{|\bar x|+1}=g, \text{ and } \big(\Phi(y)\big)_1=\big(\Phi(x)\big)_1\}$$ is finite;
\end{enumerate}

\item\label{theo:general-CHL-Theo:condition_3} If $(AAA\ldots)\in X^0_\G$ is such that $\Phi(AAA\ldots)=(ddd\ldots)\in X_\h^{inf}$, then for all $M\geq 1$ there exists a cylinder $D_{A,F}$ such that $\s^i(D_{A,F}) \subseteq C_d$ for all $i=0, 1, \ldots, M$.

\end{enumerate}

\end{theo}

\begin{proof}

Let $\Phi:X_\G\to X_\h$ be a map such that, for all $B\in\mathfrak{p}^0_{\h min}$, the set $C_B:=\Phi^{-1}(BBB\ldots)$ is a finitely defined set. Then, by Lemma~\ref{fdcontainscyl}, $\Phi$ is continuous on $X^{inf}_{\G}\cap \Phi^{-1}(\mathfrak{p}^0_{\h min})$ and hence both the forward implication and the converse of Theorem~\ref{CSC} are valid.

Suppose first that $\Phi$ is continuous and shift commuting. By Theorem~\ref{CSC}, $\Phi$ is given by $\bigl(\Phi(x)\bigr)_n=\sum_{a\in \A_\h}a\mathbf{1}_{C_a}\circ\s^{n-1}(x)$, where $\{C_a\}_{a\in \A_{\h}}$ is a partition of $X_\G$, and Items~i. and iii. above are satisfied. We need to check that $\Phi$ is a generalized sliding block code and Item~ii. above holds.

Notice that, for all $a\in\A_\h$, the sets $C_a$ and $C_a^c=\bigcup_ {b\in\A_\h\setminus\{a\}}C_b$ are unions of pseudo cylinders, which means that each $C_a$ is a finitely defined set. Hence $\Phi$ is a generalized sliding block code.

Next we check Item~ii.. Suppose that $\Phi(\bar x_1\ldots \bar x_{|\bar x|} AAA\ldots)=(BBB\ldots)\in X_\h^0$. Consider the cylinder $D_{B}$. By Theorem~\ref{CSC} (Item~ii.), there exists a cylinder $D_{(\bar x_1\ldots \bar x_{|\bar x|},A),F}$ such that $\Phi(D_{(\bar x_1\ldots \bar x_{|\bar x|},A),F})\subseteq D_{B}$. Then the finite set $F$ is such that Item~ii.a. is satisfied. To check Item~ii.b., let $x\in X_\G$ be such that $x_i=\bar x_i$ for $i=1,\ldots,|\bar x|$, $x_{|\bar x|+1}\in\varepsilon(A)$, and $\big(\Phi(x)\big)_1\in\varepsilon(B)$. Let $F=\{(\Phi(x))_1\}$. Then, by Theorem~\ref{CSC} (Item~ii.), there exists a cylinder $D_{(\bar x_1\ldots \bar x_{|\bar x|},A),F'}$  such that $\Phi(D_{(\bar x_1\ldots \bar x_{|\bar x|},A),F'})\subseteq D_{B,F}$. Hence $A_x \subseteq F'$.

For the converse, suppose that $\Phi$ is a generalized sliding block code given by $\bigl(\Phi(x)\bigr)_n=\sum_{a\in \A_\h}a\mathbf{1}_{C_a}\circ\s^{n-1}(x)$ satisfying Items~i., ii., and iii. above. All we need to do is verify Item~ii. in Theorem~\ref{CSC}.

Suppose that $\Phi(\alpha_1\ldots\alpha_k AA\ldots) = \beta_1\ldots \beta_l BB\ldots$, for some $\alpha_1\ldots\alpha_k AA\ldots \in X_{\G}^{fin}$. By Lemma~\ref{sliding block code->finite seq goes to finite seq} we have $l\leq k$. Then $\Phi(\s^l(\alpha_1\ldots\alpha_k AA\ldots)) = BB\ldots$. Denote $\s^l(\alpha_1\ldots\alpha_k AA\ldots)$ by $\bar x :=\bar x_1\ldots \bar x_{|\bar x|} AA\ldots$ (notice that $|\bar x|$ can be zero). Then $\Phi(\bar x_1\ldots \bar x_{|\bar x|} AA\ldots)=(BBB\ldots)$.

Suppose, by contradiction, that there exists a generalized cylinder $D_{B,F''}$ such that, for every generalized cylinder $D_{(\bar x_1\ldots \bar x_{|\bar x|},A),F'}$, we have that $\Phi(D_{(\bar x_1\ldots \bar x_{|\bar x|},A),F'})$ is not contained in $D_{B,F''}$.

Take $F$ as in Item~ii.a., so that $\Phi(D_{(\bar x_1\ldots \bar x_{|\bar x|},A),F})\subseteq D_{B}$. Let $x^1 \in D_{(\bar x_1\ldots \bar x_{|\bar x|},A),F}$ be such that $\Phi(x^1) \notin  D_{B,F''}$. Then $(\Phi(x^1))_{1} \in F''$. Let $D_2:= D_{(\bar x_1\ldots \bar x_{|\bar x|},A),F\cup \{(x^1)_{|\bar x| +1}\}}$, and $x^2\in D_2$ be such that  $\Phi(x^2) \notin  D_{B,F''}$ (so that $(\Phi(x^2))_{1} \in F''$). Let $D_3:= D_{(\bar x_1\ldots \bar x_{|\bar x|},A),F\cup \{(x^1)_{|\bar x| +1}, (x^2)_{|\bar x| +1}\}}$, and $x^3\in D_3$ be such that $(\Phi(x^3))_{1} \in F''$. Proceed by induction to define $x^n$, for all $n\in \N$.  Since $F''$ is finite, there exists $e\in F$ and, a subsequence $(x^{n_k})$, such that $(\Phi(x^{n_k}))_{1} = e$ for all $k$. Since the elements of $(x^n)$ are distinct this implies that $A_{x^{n_1}}$ is infinite, a contradiction. Hence Item~ii. in Theorem~\ref{CSC} is verified and it follows that $\Phi$ is continuous and shift commuting.

\end{proof}

As we mentioned before, when dealing with infinite alphabet shift spaces it is common to require that a continuous shift commuting map $\Phi:X_\G\to X_\h$ preserves length. The next corollary characterizes continuous, shift commuting, length-preserving maps.

\begin{cor}\label{length-preserving_SBC} A map $\Phi:X_\G\to X_\h$ is continuous, shift commuting, and preserves length, if and only if it is a generalized sliding block code given by $\bigl(\Phi(x)\bigr)_n=\sum_{a\in \A_\h}a\mathbf{1}_{C_a}\circ\s^{n-1}(x)$ where:
\begin{enumerate}
\item\label{cor:length-preserving_SBC:condition_1} For each $a\in \A_{\h}\setminus \mathfrak{p}^0_{\h min}$, the set $C_a$ is a (possibly empty) union of generalized cylinders of $X_\G$;

\item\label{cor:length-preserving_SBC:condition_2} $\displaystyle\bigcup_{B\in\mathfrak{p}^0_{\h min}}C_B=\mathfrak{p}^0_{\G min}$;

\item\label{cor:length-preserving_SBC:condition_3} If $\Phi(AAA\ldots) = (BBB\ldots) \in X_\h^0$ then:
 \begin{enumerate}[a.]
 \item There exists a finite subset $F \subseteq \varepsilon(A)$ such that, for all $e \in \varepsilon(A)\setminus F$, if $x\in X_\G$ and $x_{1}=e$, then $\big(\Phi(x)\big)_1=B$ or $(\Phi(x)\big)_{1}\in\varepsilon(B)$, i.e., $\Phi(x) \in D_{\color{blue}B}$;

\item For all $x\in X_\G$ with $x_{1}\in\varepsilon(A)$, and $\big(\Phi(x)\big)_1\in\varepsilon(B)$, the set $$A_x:=\{g\in\varepsilon(A): \text{ there exists } y\in X_\G \text{ with } y_{1}=g, \text{ and } \big(\Phi(y)\big)_1=\big(\Phi(x)\big)_1\}$$ is finite.
\end{enumerate}

\end{enumerate}
\end{cor}

%Before to prove the above corollary, let us notice that the condition \ref{cor:length-preserving_SBC:condition_2} means that for all $(AAA\ldots)\in X_\G^0$ we have $\Phi(AAA\ldots)=(BBB\ldots)\in X_\h^0$, and  that each empty (or non-empty?) $C_B$ is a pseudo cylinder of the form $[A]_0^0$ for some $(AAA\ldots)\in X_\G^0$. C_B pode ter mais que um elemento....pode ser A_1 uni\~{a}o A_2.

\begin{proof}%[Proof of Corollary \ref{length-preserving_SBC}]
Suppose that $\Phi$ is continuous, shift commuting and length preserving. By Proposition~\ref{prop:shift_invariant_maps} we have that $\Phi$ is given by $\bigl(\Phi(x)\bigr)_n=\sum_{a\in \A_\h}a\mathbf{1}_{C_a}\circ\s^{n-1}(x)$, where $\big\{C_a\big\}_{a\in\A_{\h}}$ is a partition of $X_\G$. Since $\Phi$ is length preserving Item~ii. above is satisfied. Furthermore, for all $B\in\mathfrak{p}^0_{\h min}$, the set $C_B:=\Phi^{-1}(BBB\ldots)$ is a countable union of elements of length zero in $X_\G$. By Example~\ref{finitexample} we have that $C_B$ is finitely defined. Items~i. and iii. now follow from Theorem~\ref{general-CHL-T}.

For the converse, let $\Phi$ be a generalized sliding block code such that Items~i. to iii. above hold. Notice that Item~ii. implies that $\Phi$ is length preserving and hence, for all $B\in\mathfrak{p}^0_{\h min}$, $\Phi^{-1}(B)$ is a finitely defined set. Now Items~i. and iii. above imply that all conditions of Theorem~~\ref{general-CHL-T} are satisfied and hence $\Phi$ is continuous and shift commuting.
\end{proof}

We end the paper presenting some examples.

\begin{ex} {\color{white}.}

\begin{enumerate}
\item[a)]
Let $\G$ be the graph with only one vertex, say $G^0:=\{w\}$, and edge set given by $\G^1:=\{d, f_1, f_2,\ldots\}$ (so all edges are loops). Let $\h$ be the graph with only one vertex, say $H^0:=\{v\}$, and edge set given by $\h^1:=\{e_1, e_2, \ldots\}$. It follows that the ultragraph shifts $X_\G$ and $X_\h$ have alphabets $\A_{\G}=\{A\}\cup\G^1$ with $A:=G^0$ and $\A_{\h}=\{B\}\cup\h^1$ with $B:=H^0$, respectively ($X_\G$ and $X_\h$ coincide with Ott-Tomforde-Willys full shifts).

  Let $C_B= [A]_1^1 \cup \{(ddd\ldots)\}$ and, for all $j$, let $C_{e_j}=[f_j]_1^1\cup[d f_j]_1^2 \cup [dd f_j]_1^3 \cup [ddd f_j]_1^4 \cup \ldots$. This partition of $X_\G$ defines a shift commuting map $\Phi$ given by $\bigl(\Phi(x)\bigr)_n=\sum_{a\in \A_\h}a\mathbf{1}_{C_a}\circ\s^{n-1}(x)$ which is not continuous (notice that $\Phi^{-1}(D_{B,\{e_1\}})$ is not open, since every open neighbourhood of $(ddd\ldots)$ contains elements of $C_{e_1}$). We remark that in this case $C_B$ is not finitely defined.

\item[b)] Let $\G$ be the graph with only one vertex, say $G^0:=\{w\}$, and edge set given by $\G^1:=\{0\}\cup\N$. Let $X_\G$ be the correspondent ultragraph shift (which, as before, has alphabet $\A_{\G}=\{A\}\cup\G^1$ with $A:=G^0$). Consider the map $\Phi:X_\G\to X_\G$ given, for all $x\in X_\G$ and $n\in\N$, by
    $$\big(\Phi(x)\big)_n=\left\{\begin{array}{lcl} x_n &\ if& x_n\neq0\ and\ x_n\neq A,\\
                                                    A  &\ if& x_n=A\ or\ x_{n+j}=0\ \forall j\geq 0,\\
                                                    k  &\ if& x_{n+j}=0\ for\ 0\leq j\leq k, \ and\ x_{n+k+1}\neq 0.
                                \end{array}\right.$$
    We have that $\Phi$ is continuous and shift commuting, but it is not a generalized sliding block code, since $C_A=[A]_1^1\cup\{(000\ldots)\}$ is not a finitely defined set.

\item[c)] In this example we consider again the ultragraph shifts of example $a)$. From Theorem \ref{general-CHL-T}, a map $\Phi:X_\G\to X_\h$, where $\Phi^{-1}(BBB\ldots)$ is a finitely defined set, is continuous and shift commuting if and only if: either $\Phi(AAA\ldots)=(BBB\ldots)$ and for all $a\in\h^1$ the set $C_a$ is a finite union of generalized cylinders; or $\Phi(AAA\ldots)=(e_je_je_j\ldots)$ for some $e_j\in\h^1$, there are just a finite number of nonempty sets $C_a$, and for all $M$ there exists a finite $F_M\subset \A_\G$ such that $\s^{n-1}(D_{A,F_M})\subset C_{e_j}$ for all $1\leq n\leq M$.

     Recall that $X_\G$ and $X_\h$ coincide with Ott-Tomforde-Willys full shifts, and therefore we can alternatively apply Theorems~3.16 and 3.17 of \cite{GSS0} to obtain the above result.

\item[d)]
In this example we use $\Z^*$ to denote the set of all non-zero integers. Let $\G$ be the ultragraph with vertex set $G^0:=\{v_k: k\geq 0 \}$, edge set $\G^1:=\{e_k: k\geq 0 \}$, and source $s_\G:\mathcal{G}^1 \to G^0$ and range $r_\G:\mathcal{G}^1 \to P(G^0)\setminus \{\emptyset\}$ maps given by
$$s_\G(e_k):=v_k,\qquad\forall k\geq 0,$$

and
$$r_\G(e_k):=\left\{\begin{array}{lcl}\{v_\ell:\ell\geq 0\}          &\ if& k=0,\\\\
                                  \{v_0,\ v_k\}                      &\ if& k\geq 1.
                \end{array}\right.$$

Note that the unique minimal infinite emitter of $\G$ is the set $A:=G^0$.

Let $\h$ be the ultragraph with vertex set $H^0:=\{w_k: k\in\Z^* \}$, edge set $\h^1:=\{f_k: k\in\Z^* \}$, and source $s_\h:\mathcal{H}^1 \to H^0$ and range $r_\h:\mathcal{H}^1 \to P(H^0)\setminus \{\emptyset\}$ maps given by
$$s_\h(f_k):=w_k \qquad\forall k\in\Z^*,$$
and
$$r_\h(f_k):=\left\{\begin{array}{lcl}\{w_{k+1}\}                       &\ if& k\leq -2,\\\\
                                  \{w_\ell:\ell\geq 1\}              &\ if& k= -1,\\\\
                                  \{w_k\}\cup\{w_\ell:\ell\leq -1\}  &\ if& k\geq 1.
                \end{array}\right.$$

We notice that the minimal infinite emitters of $\h$ are the sets $P:=\{w_\ell:\ell\leq -1\}$ and $Q:=\{w_\ell:\ell\geq 1\}$.

Now consider the map $\Phi:X_\G\to X_\h$ given, for all $x\in X_\G$ and $n\geq 1$, by

$$\big(\Phi(x)\big)_n=\left\{\begin{array}{lcl} P       &\ if& x_{n+j}=e_0\ \forall j\geq 0,\\\\
                                                f_{-k}  &\ if& x_{n+j}=e_0, \ 0\leq j\leq k-1, \text{ and } x_{n+k}\neq e_0,\\\\
                                                f_k     &\ if& x_n=e_k\ for\ k\neq 0,\\\\
                                                Q       &\ if& x_n=A.
                                \end{array}\right.$$

It follows that $\Phi$ is an invertible continuous and shift commuting map, but it is not a generalized sliding block code (since $C_P:=\{(e_0e_0e_0\ldots)\}$). On the other hand, $\Phi^{-1}$ is a generalized sliding block code.

\end{enumerate}

\end{ex}

%=======================================================================================================================

%=======================================================================================================================

\section*{Acknowledgments}

\noindent D. Gon\c{c}alves was partially supported by Conselho Nacional de Desenvolvimento Cient\'{\i}fico e Tecnol\'{o}gico -
CNPq.

\noindent M. Sobottka was supported by CNPq-Brazil PQ grant. Part of this work was carried out while the author was research visitor at Center for Mathematical Modeling, University of Chile (CNPq 54091/2017-6 and CMM CONICYT Basal program PFB 03).

%====================================================== BIBLIOGRAFIA =================================================================

\end{document}